\documentclass{article}
\usepackage[english]{babel}
\usepackage[a4paper,top=2cm,bottom=2cm,left=3cm,right=3cm,marginparwidth=1.75cm]{geometry}
\usepackage{amsmath, amsthm, tikz, amssymb}
\usepackage{graphicx}
\usepackage[colorlinks=true, allcolors=blue]{hyperref}
\usepackage{caption}
\usepackage{soul}

\usepackage{cleveref}
\usetikzlibrary{patterns,calc}

\newtheorem{theorem}{Theorem}
\newtheorem{corollary}[theorem]{Corollary}

\newtheorem{question}{Question}

\newtheorem*{claim*}{Claim}
\newtheorem*{conjecture*}{Conjecture}
\newtheorem{conjecture}{Conjecture}
\theoremstyle{definition}
\newtheorem*{definition*}{Definition}
\newtheorem{remark}[theorem]{Remark}

\newenvironment{subproof}[1][\proofname]{%
	\begin{proof}[#1]%
	}{%
	\end{proof}%
}

\newtheorem{proposition}[theorem]{Proposition}

\newcommand{\beq}[1]{\begin{equation}\label{#1}}
	\newcommand{\enq}[0]{\end{equation}}
\newcommand{\Clique}{\mathrm{Clique}}
\newcommand{\Star}{\mathrm{Star}}
\newcommand{\ep}{\varepsilon}

\title{On a clique-building game of Erd\H{o}s}
\author{Alexandru Malekshahian\footnote{Mathematical Institute, University of Oxford. E-mail: {\tt alex.malekshahian@maths.ox.ac.uk}. Research completed while the author was affiliated with the Department of Mathematics, King's College London.}\and Sam Spiro\footnote{Department of of Mathematics, Georgia State University.  Email: {\tt sspiro@gsu.edu}. This material is based upon work supported by the National Science Foundation Mathematical Sciences Postdoctoral Research Fellowship under Grant No. DMS-2202730.}}
\date{\today}

\begin{document}
	\maketitle
	
	\begin{abstract}
		The following game was introduced in a list of open problems from 1983 attributed to Erd\H{o}s: two players take turns claiming edges of a $K_n$ until all edges are exhausted. Player 1  wins the game if the largest clique that they claim at the end is strictly larger than the largest clique of their opponent; otherwise, Player 2 wins the game. Erd\H{o}s conjectured that Player 2 always wins this game for $n\geq 3$. We make the first known progress on this problem, proving that this holds for at least $3/4$ of all such $n$. We also address a biased version of this game, as well as the corresponding degree-building game, both of which were  originally proposed by Erd\H{o}s as well.
	\end{abstract}

	\section{Introduction}\label{intro}
	
	In a list of open problems from 1983, Richard Guy \cite{guy83} mentions three game-theoretic questions attributed to Erd\H{o}s, with these questions appearing jointly as problem \# 778 on Thomas Bloom's list of Erd\H{o}s problems and reiterated recently by Bloom at the 2024 British Combinatorial Conference.
	
	Each of these games (defined formally below) follows the same broad framework: two players take turns claiming previously unclaimed edges of a complete graph $K_n$ until all edges have been claimed. The winner is then declared to be the player who has claimed as their own the largest ``structure'' of a given type, and in the case of equal-sized structures, the ``weaker'' player (in a suitable sense) is declared to be the winner.
	
	All the games we study here fall into the classical paradigm of (finite) \emph{positional games}: they are two-player, deterministic games with perfect information. We quickly take stock of the standard terminology we will use. We say a game is a \emph{Player 1 (2) win} if, whenever both players play optimally, it is always Player 1 (2) who is guaranteed to win. It is a standard fact due to Zermelo \cite{Zermelo} that all positional games in which a draw is not possible are either Player 1 wins or Player 2 wins, and the central question is to identify which of the two categories a given game falls into. For a comprehensive treatise on the subject of positional games and more, we refer the reader to the classical monograph of Beck \cite{beck} as well as the more condensed book of Hefetz, Krivelevich, Stojakovi\'{c} and Szab\'{o} \cite{HKSSz}.
	
	To the best of our knowledge (and that of several colleagues we have asked), no results related to games of the type suggested by Erd\H{o}s in \cite{guy83} have been obtained until now. In this note, we take a first step towards understanding these three games of Erd\H{o}s. We begin by formally defining each of the games in turn and stating our corresponding results.   
	
	\begin{definition*} (Clique-building game)
		Given an integer $n\ge 1$, we define the two-player game Clique$(n)$ as follows. Two players take turns claiming previously unclaimed edges of $K_n$ until all edges have been claimed. They each claim precisely one edge on each of their respective turns, with Player 1 going first. Let $G_1$ and $G_2$ be the two graphs claimed by Player 1 and Player 2, respectively, at the end of the game (so that $E(G_2)=E(K_n)\setminus E(G_1)$). Player 1 wins the game if and only if they have claimed a strictly larger clique than Player 2, i.e. if $\omega(G_1)>\omega(G_2)$ where $\omega(G)$ denotes the clique number of a graph $G$. Otherwise, Player 2 wins.
	\end{definition*}
	
	\begin{conjecture*}[Erd\H{o}s \cite{guy83}]
		For each $n\geq 3$, Clique$(n)$ is a Player 2 win.
	\end{conjecture*}
	
	Let us first give some background and motivation for this problem. A classical argument of \emph{strategy stealing} originally due to Nash shows that Player 1 can always guarantee at least $\omega(G_1)\geq \omega(G_2)$. Indeed, suppose that Player 2 had a strategy $S$ which guarantees that $\omega(G_1)<\omega(G_2)$. Then Player~1 could make an initial, arbitrary move and on subsequent turns ignore their initially claimed edge and instead pretend they were the second player; thus, following the strategy $S$ they can also guarantee that $\omega(G_1)>\omega(G_2)$, a contradiction (see also \cite{beck,HKSSz} for details).
	
	Therefore, the most that Player 2 could hope for in this sort of clique-building game is to guarantee a largest clique of size equal to that of Player 1 but no larger. Strategy stealing thus formalizes the intuition that the player who goes first in these kinds of games should have an advantage with regards to building a `target' structure. We think of the above conjecture as precisely saying that this advantage is negligible: while Player 1 can always do at least as well as Player 2 (in terms of size of the largest clique), they should not be able to do \emph{strictly} better.
	
	While we are unable to prove this result for all values of $n$, we are able to show that it holds for most values.
	
	\begin{theorem}\label{clique}
		Let $\mathcal{C}\subseteq \mathbb{N}$ be the set of all natural numbers $n$ for which the Clique$(n)$ game in a Player 2 win. Then $\mathcal{C}$ has asymptotic density at least $\frac{3}{4}$.
	\end{theorem}
	We will prove \Cref{clique} by a strategy stealing argument.  More precisely, we will show that if Player 1 wins for a given value $n$, then this will imply that Player 2 wins at the values $n+1,\ n+2$, and $n+3$; see \Cref{sec:clique} for more.

	We now turn to our second game.
	\begin{definition*}(Star-building game)
		Given an integer $n\ge 1$, we define the two-player game Star$(n)$ as follows. As in the Clique$(n)$ game, Players 1 and 2 alternate claiming previously unclaimed edges of $K_n$ until all edges have been claimed. Player 1 wins if and only if, after all edges have been claimed, the largest star they have claimed is strictly larger than the largest star claimed by Player 2, i.e. if $\Delta(G_1)>\Delta(G_2)$. Otherwise, Player 2 wins.
	\end{definition*}
	It was asked in \cite{guy83} who wins this game, and as in the case of Clique$(n)$, we manage to show that it is Player 2 a majority of the time.
	
	\begin{theorem}\label{star}
		Let $\mathcal{S}$ be the set of all natural numbers $n$ for which the Star$(n)$ game is a Player 2 win. Then $\mathcal{S}$ has asymptotic density at least $\frac{2}{3}$.
	\end{theorem}
	Again, we will prove \Cref{star} by a strategy stealing argument by showing that if Player 1 wins for a given value $n$, then this will imply that Player 2 wins at the values $n+1$ and $n+2$.
	
	We now consider our final game, which is a generalization of Clique$(n)$.
	\begin{definition*}(Biased clique-building game)
		Given integers $n,p,q\ge 1$, we define the two-player game Clique$_{(p,q)}(n)$ as follows. Players 1 and 2 alternate claiming previously unclaimed edges of $K_n$, with Player 1 going first. Player 1 claims $p$ edges on each of their turns, while Player 2 claims $q$ edges on each of their turns (unless it is the beginning of Player~1's next turn, and there are fewer than $p$ edges remaining, or the beginning of Player~2's turn and there are fewer than $q$ edges remaining, in which case the player to move simply claims all remaining unclaimed edges). If $p\ge q$, then Player 1 wins if and only if their largest clique is strictly larger than the largest clique of their opponent, i.e.\ if $\omega(G_1)>\omega(G_2)$, and otherwise Player 2 wins.  If $p<q$, then Player 2 wins if and only if their largest clique is strictly larger than the largest clique of their opponent, i.e.\ if $\omega(G_2)>\omega(G_1)$, and otherwise Player 1 wins.
	\end{definition*}
	
	Note that when $p=q=1$ this game exactly recovers Clique$(n)$, and in general this game follows the philosophy that the ``stronger'' player (i.e.\ the one who claims strictly more edges on their turn) needs to make a strictly larger clique at the end to be declared the winner.

	It was suggested in \cite{guy83} that Player 2 can always win the game $\Clique_{(1,2)}(n)$ for $n\ge 4$, i.e.\ that the advantage of Player 1 going first in the game $\Clique(n)$ is completely eliminated if Player 2 gets to claim more edges with each of their moves. While we are unable to show Player 2 wins $\Clique_{(1,q)}(n)$ for $q=2$, we are able to do this for a somewhat weaker value of $q$.  Indeed, the original version of this paper gave a simple proof showing Player 2 wins $\Clique_{(1,q)}(n)$ for $q=15$ for large $n$, and somewhat more generally that Player 2 wins $\Clique_{(p,q)}(n)$ whenever $q\ge 16^p$ and $n$ is large.  After submitting our paper, a referee pointed out a substantial refinement of our argument, allowing us to prove the following.

	\begin{theorem}\label{bias}
		For all $p\ge 1$, there exists an integer $n_0$ such that Player 2 wins Clique$_{(p,q)}(n)$ for all $n\ge n_0$ provided
		\[q\ge \left(2+\frac{12 \log_2\log_2(p+1)}{\log_2(p+1)}\right)p+2.\]
		In particular, for all sufficiently large $n$, the game Clique$_{(1, 4)}(n)$ is a Player 2 win.
	\end{theorem}
	That is, Player 2 can win Clique$_{(p,q)}(n)$ for $q=(2+o(1))p$.  This bound appears to be quite close to best possible, as we believe that Player 2 cannot win if $q<p$ -- see \Cref{conj1}.
	
	
	As a final aside, we note that the above games can all be viewed as examples of \emph{Maker-Breaker} games (see \cite{beck,HKSSz} for formal definitions). Indeed, the winning sets for Maker are precisely the graphs $G\subseteq K_n$ which satisfy $\omega(G)>\omega(\overline G)$ (in the case of Clique$(n)$ and Clique$_{(p,q)}(n)$) or $\Delta(G)>\Delta(\overline G)$ (in the case of Star$(n)$). That being said, there is a clear departure of the games studied here from the spirit of more traditional Maker-Breaker games whose winning sets typically have a simpler structure, such as being connected or Hamiltonian. In particular, the `self-dual' nature of the games proposed by Erd\H{o}s makes it difficult to apply classical tools from the study of Maker-Breaker games. Indeed, we could only make use of Beck's method of `building by blocking' and the Erd\H{o}s-Selfridge criterion in our analysis of the the Clique$_{(p, q)}(n)$ game for $q$ large relative to $p$. We therefore need to employ different methods in our proofs of \Cref{clique,star}. 
	
	\begin{remark}\label{monotonicity}
		In what follows, we will make heavy use of the \emph{monotonicity} of all of the above games. Namely, if the current state of the game is a Player $i$ win and we decide to give a set of additional edges to Player $i$ (thus breaking the usual pattern of players alternating turns), then the game stays a win for Player $i$.
	\end{remark}
	
	\section{The clique-building game}\label{sec:clique}
	
	In this section we prove \Cref{clique}, showing that the clique-building game $\Clique(n)$ is a Player 2 win for at least $3/4$ of the possible values of $n$. As mentioned in the introduction, we will prove the following stronger statement: if $n$ sufficiently large is such that Clique$(n)$ is a Player 1 win, then Clique$(n+1)$, Clique$(n+2)$ and Clique$(n+3)$ are all Player 2 wins, which immediately gives \Cref{clique}. We prove each part of this stronger result as separate propositions.
	
	\begin{proposition}\label{clique+1}
		If  Clique$(n)$ is a Player 1 win for some $n\geq 2$, then Clique$(n+1)$ is a Player 2 win.
	\end{proposition}
	\begin{proof}
		Suppose $n\geq 2$ is such that Clique$(n)$ is a Player 1 win. We use a strategy stealing argument, in the sense that in order for Player 2 to win the Clique$(n+1)$ game, they will make use of the winning strategy of \emph{Player~1} from the Clique$(n)$ game. To avoid confusion, we find it helpful at this stage to introduce names for our two players, namely Red and Blue, where we use ``she" to refer to Red and ``he" for Blue.  Specifically, we will consider the Clique$(n+1)$ game with Red going first and will show that Blue has a winning strategy (while playing as the second player).
		
		Red can only ever begin the Clique$(n+1)$ game by claiming some arbitrary edge, say $uv$. By monotonicity of the game (see \Cref{monotonicity}), Blue may at this point choose to give more edges to his opponent. In particular, he now chooses to give all the edges $\{vw : \ w\in V(K_{n+1}) \}$ to Red. At this stage, the remaining edges that have not been claimed by either player form precisely the edge set of a $K_n$, and it is Blue's turn to move (see \Cref{fig1}). By our assumption on $n$, we know that Blue has a winning strategy in the Clique$(n)$ game as the \emph{first} player, i.e. he can build a strictly larger clique than his opponent inside this $K_n$. Blue now follows this strategy until the end of the game.
		
		\begin{figure}[htb]
			\captionsetup{justification=centering}
			\centering
			\begin{tikzpicture}[scale=0.8, cross/.style={path picture={ 
						\draw[black]
						;
				}}]
				\def\a{1.5} 
				\def\b{2} 
				\def\q{4} 
				\def\x{{\a^2/\q}} 
				\def\y{{\b*sqrt(1-(\a/\q)^2}} 
				\coordinate (O) at (0,0); 
				\coordinate (Q) at (-\q,0); 
				\coordinate (P) at (-\x,\y); 
				\coordinate (P') at (-\x,-\y);
				
				\draw[thick] (O) ellipse({\a} and {\b});
				\draw[red,thick] ($(Q)!-0!(P)$) -- ($(Q)!1!(P)$);
				\draw[red,thick] ($(Q)!-0!(P')$) -- ($(Q)!1!(P')$);
				\draw[red, thick] (-4, 0) -- (-1.14, 1.3);
				\draw[red, thick] (-4, 0) -- (-1.35, 0.9);
				\fill[] (Q) circle(0.05) node[below left] {v};
				\node at (-2.5, 0) {\dots};
				
				\node at (0, 0) {$K_n$};
				
			\end{tikzpicture}
			\caption{The game state before Blue claims his first edge. We are guaranteed a strictly larger Blue clique than Red inside the $K_n$ subgraph indicated above when Blue follows Player 1's winning strategy from $\Clique(n)$.}\label{fig1}
		\end{figure}
		
		Let $G_R,G_B$ be the graphs claimed by Red and Blue in $K_{n+1}$ at the end of this game, and let $G'_R,G'_B$ be the subgraphs they claim in the $K_n$ subgraph $K_{n+1}\setminus\{v\}$.  Because Blue used the winning strategy for $\Clique(n)$ in $K_{n+1}\setminus\{v\}$, we must have $\omega(G'_B)>\omega(G'_R)$.  Note also that $\omega(G_R)=\omega(G'_R)+1$ since $v$ is adjacent to every vertex of $G'_R$ by construction, and also $\omega(G_B)=\omega(G_B')$.  We conclude that 
		\[\omega(G_B)=\omega(G_{B}')\ge \omega(G_{R}')+1=\omega(G_R),\]
		and hence Blue wins the Clique$(n+1)$ game as Player 2 regardless of how Red plays.
	\end{proof}
	
	We next prove that if Clique$(n)$ is a Player 1 win for some $n\geq 2$, then Clique$(n+2)$ is a Player 2 win.  A referee pointed out to us that there is in fact a very simple proof of this result which we sketch here: 
	Below we present our 
	
	\begin{proposition}\label{clique+2}
		If  Clique$(n)$ is a Player 1 win for some $n\geq 2$, then Clique$(n+2)$ is a Player 2 win.
	\end{proposition}
	After submitting our paper, a referee pointed out that there is in fact a relatively simple proof of this result which we sketch now: if Player 1 starts with some edge $xy\in E(K_{n+2})$, then Player 2 can use a strategy stealing argument on the remaining $n$ vertices except when Player 1 picks an edge of the form $xz$ or $yz$, in which case Player 2 picks the edge $yz$ or $xz$, respectively.  With this, $x,y$ can not belong to a largest clique of Player 1 (unless this largest clique has size 2, in which case they will lose the game), and hence the largest clique of Player 1 is at most one more than their largest clique in $K_{n+2}\setminus \{x,y\}$, which by the strategy stealing argument is in total at most that of Player 2's largest clique.
	
	Below we present our original (and more complicated) proof of \Cref{clique+2}, largely because this serves as a warmup to the more complicated argument we make in \Cref{clique+3}.
	
	\begin{proof}
		The result is straightforward to prove for $n=2$, so suppose $n\ge 3$ is such that Clique$(n)$ is a Player 1 win.  We will prove $\Clique(n+2)$ is a Player 2 win by making use of a strategy stealing argument (similar to the previous proof) together with a case analysis of Player 1's second move.  As before, we introduce the names Red and Blue for the two players and show that if Red plays first in the Clique$(n+2)$ game, then Blue has a winning strategy (as Player 2) in this game.
		
		Consider the configuration $S_1$ in the Clique$(n+2)$ game, where $v_1$ and $v_2$ have all the edges to the $K_n$ present in red, as depicted in \Cref{fig2}: 
		\begin{figure}[htb]
			\captionsetup{justification=centering}
			\centering
			\begin{tikzpicture}[scale=0.8, cross/.style={path picture={ 
						\draw[black]
						;
				}}]
				\def\a{1.5} 
				\def\b{2} 
				\def\q{4} 
				\def\x{{\a^2/\q}} 
				\def\y{{\b*sqrt(1-(\a/\q)^2}} 
				\coordinate (O) at (-1,0); 
				\coordinate (Q) at (-\q,-1); 
				\coordinate (Q') at (-\q, 1);
				\coordinate (P) at (-\x,\y); 
				\coordinate (P') at (-\x,-\y);
				
				\draw[thick] (O) ellipse({\a} and {\b});
				\draw[red, thick] (-4, 1) -- (-3.2, 1.2);
				\draw[red, thick] (-4, 1) -- (-3.2, 1);
				\draw[red, thick] (-4, 1) -- (-3.2, 0.8);
				\node at (-3.2, 0.6) {\dots};
				\draw[red, thick] (-4, -1) -- (-3.2, -1.2);
				\draw[red, thick] (-4, -1) -- (-3.2, -1);
				\draw[red, thick] (-4, -1) -- (-3.2, -0.8);
				\node at (-3.2, -0.6) {\dots};
				\draw[red, thick] (-4, 1) -- (-3.2, 0.2);
				\draw[red, thick] (-4, -1) -- (-3.2, -0.2);
				\fill[] (Q) circle(0.05) node[below left] {$v_1$};
				\fill[] (Q') circle(0.05) node[below left] {$v_2$};
				\draw[blue, thick] (-4, -1) -- (-4, 1);
				\fill[] (-1.3, -1.3) circle(0.05) node[below left] {$u_1$};
				\fill[] (-0.9, -0.8) circle(0.05) node[below right] {$u_2$};
				\draw[blue, thick] (-1.3, -1.3) -- (-0.9, -0.8);
				
				\node at (-1, 0) {$K_n$};
				
			\end{tikzpicture}
			\caption{The position $S_1$.}\label{fig2}
		\end{figure}
		
		\begin{claim*}
			If Red and Blue alternate claiming one edge starting from $S_1$ with Red going first, then Blue has a strategy to end the game with a clique of size at least as large as that of Red.
		\end{claim*}
		\begin{subproof}
			Note that the restriction of $S_1$ to the $K_n$ subgraph $K_{n+2}\setminus\{v_1, v_2\}$ pictured above is precisely an instance of the Clique$(n)$ game in which Blue has made the first move, and no other move has been played yet. By our assumption that Clique$(n)$ is a Player 1 win, it follows that Blue has a strategy which guarantees him a strictly larger clique than any of Red's cliques inside the $K_n$ subgraph at the end of the game. Re-adding the vertices $v_1$ and $v_2$ to this $K_n$ and noting that the edge $v_1v_2$ is blue, we see that Red's largest clique can increase in size by at most 1 in $K_{n+2}$, so it can be at most equal to Blue's largest clique.   
		\end{subproof}
		
		It remains to show that Blue, playing second starting from the empty board on $K_{n+2}$, can always arrive at state $S_1$. Red necessarily starts by playing an arbitrary edge $u_1v_1$ of $K_{n+2}$. Blue's strategy then is to play an edge $u_1u_2$ incident to the one Red has just played. There are then 5 cases (up to isomorphism) for Red's next move as depicted below:
		\begin{figure}[htb]
			\captionsetup{justification=centering}
			\begin{center}
				\begin{tikzpicture}
					\draw[red, thick] (1, 0) -- (0, 0) -- (0, 1);
					\draw[blue, thick] (0, 1) -- (1, 1);
					\fill[] (0, 1) circle(0.05) node[left] {$u_1$};
					\fill[] (1, 1) circle(0.05) node[right] {$u_2$};
					\fill[] (0, 0) circle(0.05) node[left] {$v_1$};
					\fill[] (1, 0) circle(0.05) node[right] {};
					
					\draw[red, thick] (3, 1) -- (3, 0) -- (4, 1);
					\draw[blue, thick] (4, 1) -- (3, 1);
					\fill[] (3, 1) circle(0.05) node[left] {$u_1$};
					\fill[] (4, 1) circle(0.05) node[right] {$u_2$};
					\fill[] (3, 0) circle(0.05) node[left] {$v_1$};
					
					\draw[red,thick] (6, 1)--(6, 0);
					\draw[red, thick] (7, 1)--(7, 0);
					\draw[blue,thick] (6,1)--(7,1);
					\fill[] (6, 1) circle(0.05) node[left] {$u_1$};
					\fill[] (7, 1) circle(0.05) node[right] {$u_2$};
					\fill[] (6, 0) circle(0.05) node[left] {$v_1$};
					\fill[] (7, 0) circle(0.05) node[right] {};
					
					\draw[red, thick] (9, 0) -- (9, 1);
					\draw[red,thick] (10, 0) -- (9, 1);
					\draw[blue, thick] (10, 1)-- (9, 1);
					\fill[] (9, 1) circle(0.05) node[left] {$u_1$};
					\fill[] (10, 1) circle(0.05) node[right] {$u_2$};
					\fill[] (9, 0) circle(0.05) node[left] {$v_1$};
					\fill[] (10, 0) circle(0.05) node[right] {};
					
					\draw[red, thick] (13.7, 0) -- (13.7, 1);
					\draw[red, thick] (12,1) -- (12,0);
					\draw[blue, thick] (12, 1) -- (13, 1);
					\fill[] (12, 1) circle(0.05) node[left] {$u_1$};
					\fill[] (13, 1) circle(0.05) node[right] {$u_2$};
					\fill[] (12, 0) circle(0.05) node[left] {$v_1$};
					\fill[] (13.7, 0) circle(0.05) node[right] {};
					\fill[] (13.7, 1) circle(0.05) node[right] {};
				\end{tikzpicture}
			\end{center}
			\caption{The five possible game states after Red's second move given that Blue played his first move incident to Red's first move.}\label{fig:5States}
		\end{figure}
		
		It is not difficult to see that each of these game states are (coloured) subgraphs of $S_1$; indeed, one can see this by taking $v_2$ to be an isolated vertex in the first and second pictures (which exists since $n+2\ge 5$) and by taking $v_2$ to be any of the unlabelled vertices in the last three pictures. Indeed, in all five of these cases, every red edge is incident to precisely one of $v_1$ and $v_2$.  With this, we see that Blue can play his next move by adding a blue edge between these two vertices identified as $v_1,v_2$ and then strengthen his opponent by giving up to Red all the remaining edges needed in order to arrive at exactly state $S_1$.  This proves that Blue can always reach position $S_1$, and the result follows by the above claim.
	\end{proof}
	
	\begin{proposition}\label{clique+3}
		If  Clique$(n)$ is a Player 1 win for some $n\geq 25$, then Clique$(n+3)$ is a Player 2 win.
	\end{proposition}
	\begin{proof}
		We use strategy stealing together with a case analysis of Player 1's second move and a mirroring strategy.  Suppose $n$ is such that Clique$(n)$ is a Player 1 win. As usual, we will play the Clique$(n+3)$ game with Red going first and show that Blue has a winning strategy.  However, unlike in the previous two proofs, we will not be able to guarantee that we reach our ``good'' game state $S$ after some fixed number of moves, and in particular, we will need to allow for the $K_n$ subgraph of $K_{n+3}$ to be essentially any graph that comes about from Blue playing the winning strategy as Player 1 in $\Clique(n)$.
		
		With this in mind, fix some arbitrary winning strategy $\mathcal{S}$ of Player 1's in $\Clique(n)$ and define a 2-coloured subgraph $H\subseteq K_n$ to be \emph{good} if (1) $H$ occurs as a game state at some point during a run of the Clique$(n)$ game in which Blue goes first and follows $\mathcal{S}$ for $\Clique(n)$; and (2) $H$ is such that either Blue has claimed one more edge than Red or all edges have been claimed.
		
		For such a graph $H$, we define the following game state $S_2(H)$ as a (coloured) subgraph of $K_{n+3}$, in which each of the vertices $v_1$ and $v_2$ have the same three red neighbors in $H$ and $u$ has red degree equal to $n$ into $H$ - see \Cref{fig4}.
		
		\begin{figure}[htb]
			\captionsetup{justification=centering}
			\centering
			\begin{tikzpicture}
				
				\draw[red,thick] (1, 1)--(1, -1);
				\draw[blue,thick] (1,1)-- (0, 0)--(1, -1);
				\draw[red, thick] (0, 0) -- (0.5, 0.3);
				\draw[red, thick] (0, 0) -- (0.5, 0.15);
				\node at (0.5, 0) {$\dots$};
				\draw[red, thick] (0, 0) -- (0.5, -0.3);
				
				\draw[thick,black,pattern=north east lines,pattern color=gray] (3.5, 0) ellipse({1} and {1.8});
				
				\draw[red, thick] (1, 1) -- (2.55, 0.5);
				\draw[red, thick] (1, -1) -- (2.55, 0.5);
				\fill[] (2.55, 0.5) circle(0.05) {};
				
				\draw[red, thick] (1, 1) -- (2.55, -0.5);
				\draw[red, thick] (1, -1) -- (2.55, -0.5);
				\fill[] (2.55, -0.5) circle(0.05) {};
				
				\draw[red, thick] (1, 1) -- (2.5, 0);
				\draw[red, thick] (1, -1) -- (2.5, 0);
				\fill[] (2.5, 0) circle(0.05) {};
				
				\node at (3.5, 0) {$H$};
				\fill[] (0, 0) circle(0.05) node[above left] {u};
				\fill[] (1, 1) circle(0.05) node[above left] {$v_1$};
				\fill[] (1, -1) circle(0.05) node[below left] {$v_2$};
			\end{tikzpicture}
			\caption{The position $S_2(H)$.}\label{fig4}
		\end{figure}
		
		\begin{claim*}
			If $H$ is a good state and if Red and Blue alternate claiming one edge starting from $S_2(H)$ with Red going first, then Blue has a strategy to end the game with a clique of size at least as large as that of Red.
		\end{claim*}
		\begin{subproof}
			Blue plays as follows. Every time Red makes a move inside the subgraph $K_n=K_{n+3}\setminus\{ u, v_1, v_2\}$, Blue responds according to the winning strategy $\mathcal{S}$ of Player 1 for the Clique$(n)$ game (which he can do, since $H$ is a good state). If Red has just claimed the last edge of $K_n$, then Blue skips his turn (which is allowed by \Cref{monotonicity}). Whenever Red instead makes a move of the type $v_iw$ for some $i\in\{1, 2\}$ and $w\in K_n$, Blue responds by playing $v_{3-i}w$. This sequence of play continues until all edges of $K_{n+3}$ have been claimed. 
			
			Let $G_R$ and $G_B$ be the graphs claimed by Red and Blue respectively at the end of this game, and let $G_R'$ and $G_B'$ be the subgraphs claimed by Red and Blue on the vertex set $V(K_{n+3})\setminus\{u, v_1, v_2\}$ respectively.  By assumption, we know that \[\omega(G_B')\ge \omega(G_R')+1.\]  By definition of $S_2(H)$ and the strategy used above, $v_1,v_2$ will end with at most three common neighbors in $G_R$, and therefore
			\[\omega(G_R)\le \max\{\omega(G_R')+1,5\}\]
			since any clique in $G_R$ must either have size at most 5 or use at most one of the three vertices $u,v_1,v_2$.  Finally, because the Ramsey number $R(4,5)$ is well known to equal $25\le n$, we must have \[\omega(G_B)\geq\omega(G_B')=\max\{\omega(G_B'),\omega(G_R')\}\geq 5.\]
			
			Combining the three inequalities above yields $\omega(G_B)\ge \omega(G_R)$, proving the result.
		\end{subproof}
		
		Using the above claim, it remains to show that Blue can always arrive in the state $S_2(H)$ for some good $H$ while playing the Clique$(n+3)$ game as Player 2.  As in the proof of \Cref{clique+2}, Blue will play his first edge $xy$ incident to the first edge that Red plays, implying that after Red's second move, the game will be in one of the five (unlabeled) states depicted in \Cref{fig:5States}.  We aim to show that in each of these possible states, Blue can play in such a way that the game can eventually be embedded into $S_2(H)$ for some good $H$. To this end, we define $z$ to be any vertex which is not in a red edge with either $x$ or $y$ and which is incident to every red edge not containing $x$ or $y$; see \Cref{fig5} below for how $z$ is defined in each of the five possible states.
		
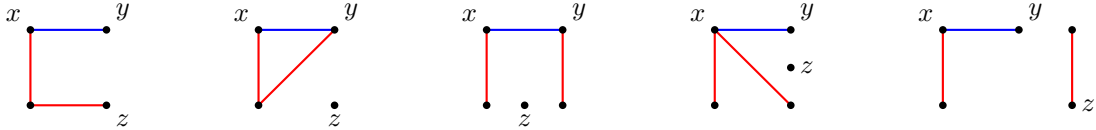
\begin{figure}[htb]
			\captionsetup{justification=centering}
			\begin{center}
				\begin{tikzpicture}
					\draw[blue, thick] (1, 1)-- (0, 1);
					\draw[red, thick] (0,1)-- (0, 0) -- (1, 0);
					\fill[] (1, 1) circle(0.05) node[above right] {$y$};
					\fill[] (0, 1) circle(0.05) node[above left] {$x$};
					\fill[] (1, 0) circle(0.05) node[below right] {$z$};
					\fill[] (0, 0) circle(0.05) node[below right] {};
					\draw[red, thick] (3, 1) -- (3, 0) -- (4, 1);
					\draw[blue, thick] (4, 1) -- (3, 1);
					\fill[] (3, 1) circle(0.05) node[above left] {$x$};
					\fill[] (4, 1) circle(0.05) node[above right] {$y$};
					\fill[] (4, 0) circle(0.05) node[below] {$z$};
					\fill[] (3, 0) circle(0.05) node[below right] {};
					
					\draw[red,thick] (6, 1)--(6, 0);
					\draw[red, thick] (7, 1)--(7, 0);
					\draw[blue,thick] (6,1)--(7,1);
					\fill[] (6, 1) circle(0.05) node[above left] {$x$};
					\fill[] (7, 1) circle(0.05) node[above right] {$y$};
					\fill[] (6.5, 0) circle(0.05) node[below] {$z$};
					\fill[] (6, 0) circle(0.05) node[below right] {};
					\fill[] (7, 0) circle(0.05) node[below right] {};
					
					\draw[red, thick] (10, 0) -- (9, 1);
					\draw[red,thick] (9, 0) -- (9, 1);
					\draw[blue, thick] (10, 1)-- (9, 1);
					\fill[] (9, 1) circle(0.05) node[above left] {$x$};
					\fill[] (10, 1) circle(0.05) node[above right] {$y$};
					\fill[] (9, 0) circle(0.05) node[below] {};
					\fill[] (10, 0) circle(0.05) node[below right] {};
					\fill[] (10, 0.5) circle(0.05) node[right] {$z$};
					
					\draw[red, thick] (13.7, 0) -- (13.7, 1);
					\draw[red, thick] (12,1) -- (12,0);
					\draw[blue, thick] (12, 1) -- (13, 1);
					\fill[] (12, 1) circle(0.05) node[above left] {$x$};
					\fill[] (13, 1) circle(0.05) node[above right] {$y$};
					\fill[] (12, 0) circle(0.05) node[] {};
					\fill[] (13.7, 0) circle(0.05) node[right] {$z$};
					\fill[] (13.7,1) circle(-.05) node[] {};
					
				\end{tikzpicture}
				\caption{The five possible game states depicted in \Cref{fig:5States} together with an appropriate choice for the vertex $z$.}\label{fig5}
			\end{center}
		\end{figure}
		
		The vertices denoted by $x, y$ and $z$ will ultimately play the roles of $v_1, v_2$ and $u$, in some order to be decided later. Blue's strategy from here onwards is as follows: play as the first player in the Clique$(n)$ game on the sub-board $K_{n+3}\setminus\{x, y, z\}$ (skipping his turn if Red just claimed the last edge of this $K_n$) until Red claims some edge $e$ incident to (at least) one of the vertices $x,y,z$.  From here, Blue claims one of the two edges $xz$ or $yz$ (which is possible since at least one of these edges does not equal $e$ and since neither of these edges were claimed earlier, by how we defined $z$ not to be in any edge with $x,y$). 
		He then labels whichever vertex in $\{x,y\}$ has blue degree 2 as $u$ and the other two vertices in $\{x,y,z\}$ as $v_1,v_2$, and finally Blue gives up to Red all the remaining edges that are necessary to arrive at state $S_2(H)$, where $H$ is defined to be the sub-board $K_{n+3}\setminus\{x, y, z\}$ before edge $e$ was claimed.  Note that this last step is possible since by construction, the set $\{x,y,z\}$ sees at most 3 total red edges after $e$ is claimed, and hence this set in total has at most 3 red edges into $K_n$.  This establishes that Blue can always arrive at $S_2(H)$ for some good $H$, proving the result by the claim.
	\end{proof}
	
	\begin{remark}
		It is possible to extend \Cref{clique+3} to all $n\geq 2$.  For example, one can easily achieve $n\ge 9$ by noticing that we can always achieve a state $S_2'(H)$ in which the vertices $v_1$ and $v_2$ only share \emph{two} red neighbors in $H$ together with the fact that $R(3,4)=9$. Indeed, in the very last step of the proof, we can have Blue claim one of $xz$ and $yz$ in such a way that there are at most two red edges between $\{x, y, z\}$ and the other vertices.  In fact, one can even ensure that $v_1,v_2$ either have a blue edge between them or at most \textit{one} red neighbor by being slightly more careful with the strategy used to push the bound all the way to $n\ge 2$, but this approach requires us to deviate from the notation of the current proof.  We choose not to give all the details of this analysis here for ease of reading. 
	\end{remark}
	
	We now formally put these propositions together to prove our main result.
	\begin{proof}[Proof of \Cref{clique}]
		By the propositions above, for each integer $n\ge 25$, at least three of the numbers in $\{n,n+1,n+2,n+3\}$ are such that $\Clique(n)$ is a Player 2 win.  It follows then that the set of winning values for Player 2 has asymptotic density at least $3/4$.
	\end{proof}
	
	\section{The star-building game}
	
	Similarly to the clique-building game $\Clique(n)$, we establish \Cref{star} by proving that if Star$(n)$ is a Player 1 win for some $n\geq3$, then Star$(n+1)$ and Star$(n+2)$ are Player 2 wins.
	
	\begin{proposition}\label{star+1}
		If Star$(n)$ is a Player 1 win for some $n\geq 3$, then Star$(n+1)$ is a Player 2 win.    
	\end{proposition}
	\begin{proof}
		Suppose $n$ is such that Star$(n)$ is a Player 1 win. We play the Star$(n+1)$ game with Red going first and show that Blue has a winning strategy as Player 2.  Call Red's first edge $uv$. Blue then chooses to play some arbitrary edge not incident to either $u$ nor $v$, which exists since $n+1\ge 4$ - see \Cref{fig6}.
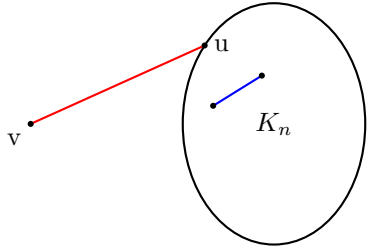
		\begin{figure}[htb]
			\captionsetup{justification=centering}
			\begin{center}
				\begin{tikzpicture}[scale=0.8, cross/.style={path picture={ 
							\draw[black]
							;
					}}]
					\def\a{1.5} 
					\def\b{2} 
					\def\q{4} 
					\def\x{{\a^2/\q}} 
					\def\y{{\b*sqrt(1-(\a/\q)^2}} 
					\coordinate (O) at (0,0); 
					\coordinate (Q) at (-\q,0); 
					\coordinate (P) at (-\x,\y); 
					\coordinate (P') at (-\x,-\y);
					
					\draw[thick] (O) ellipse({\a} and {\b});
					
					\draw[red, thick] (-4, 0) -- (-1.14, 1.3);
					\draw[blue, thick] (-1,0.3) -- (-0.2, 0.8);
					\fill[] (-1, 0.3) circle(0.05);
					\fill[] (-0.2, 0.8) circle(0.05){};
					\fill[] (Q) circle(0.05) node[below left] {v};
					\fill[] (-1.14, 1.3) circle(0.05) node[right] {u};
					\node at (0, 0) {$K_n$};
					
				\end{tikzpicture}
				\caption{The game state after Blue claims his first edge.}\label{fig6}
			\end{center}
		\end{figure}
		
		From here on, Blue plays as follows. Whenever Red makes a move inside the $K_n$ subgraph not containing the vertex $v$, then Blue responds according to the winning strategy of Player 1 in the Star$(n)$ game. If Red has just claimed the last free edge of the $K_n$ subgraph, then Blue simply skips his turn. Whenever Red instead plays an edge incident to $v$, Blue responds by picking another, unclaimed edge also incident to $v$ arbitrarily (unless all edges incident to $v$ have been claimed, in which case Blue again skips his turn). We claim that this is a winning strategy for Blue as the second player.
		
		In order to see this, let $G_R,G_B$ be the graphs claimed by these two players at the end of the game and $G_R',G_B'$ the graphs they claim in $K_{n+1}\setminus\{v\}$.  The two key observations that we will need about our strategy outlined above are
		\begin{equation}\label{eq:vDeg}\deg_{G_R}(v)=1+\lceil (n-1)/2\rceil,\end{equation}
		since Red claims the first edge on $v$ together with half (rounded up) of the remaining $n-1$ edges incident to $v$; and that
		\[\Delta(G'_B)>\Delta(G'_R),\]
		because Blue plays in such a way that he wins the Star$(n)$ game played on the $K_n$ sub-board as Player 1.

		Let $w$ be a vertex of $K_n$ (i.e.\ not $v$) such that $\deg_{G_B'}(w)=\Delta(G'_B)$.  The inequality above then implies that
		\[\deg_{G_B}(w)\ge \deg_{G'_B}(w)\ge \max_{x\in K_n}\deg_{G'_R}(x)+1\ge \max_{x\in K_n}\deg_{G_R}(x),\]
		with this last step using that the degree of any $x\in K_n$ can increase by at most 1 from $G_R'$ to $G_R$ due to $v$.  With this, we see that the only possible way for Blue to lose is if we have $\deg_{G_R}(v)>\deg_{G_B}(w)$. To show that this cannot happen, we distinguish two cases based on the parity of $n$.
		
		\textit{Case 1: $n=2k$.}  By \eqref{eq:vDeg} we have  $\deg_{G_R}(v)=k+1$. However, we cannot have $\deg_{G_B'}(w)\leq k$, since this would imply that every vertex of the $K_n$ sub-board sees precisely $k$ blue edges and $k-1$ red edges (since otherwise Red would have a vertex of degree at least $\deg_{G_B'}(w)=\Delta(G_B')$).  In total this would mean that the number of red edges claimed in this sub-board is only $\frac{1}{2}(k-1)n=(k-1)k$, contradicting the fact that the number of red edges is exactly \[\left\lfloor \frac{1}{2} {n\choose 2}\right\rfloor=\left\lfloor \frac{1}{2}(2k-1)k\right\rfloor>(k-1)k\]
		with this last inequality holding for all $k\ge 2$.  We conclude that $\deg_{G'_B}(w)\ge k+1=\deg_{G_R}(v)$ as desired.
		
		\textit{Case 2:} $n=2k+1$. In this case $\deg_{G_R}(v)=k+1$, and again we cannot have $\deg_{G'_B}(w)\leq k$ as this would imply $\deg_{G_B'}(w)\le \deg_{G_R'}(w)$, a contradiction to the fact that $\deg_{G_B'}(w)>\deg_{G_R'}(w)$.  We again conclude that $\deg_{G'_B}(w)\ge k+1=\deg_{G_R}(v)$, proving the result.
	\end{proof}
	
	\begin{proposition}\label{star+2}
		If Star$(n)$ is a Player 1 win for some $n\geq 3$, then Star$(n+2)$ is a Player 2 win.      
	\end{proposition}
	\begin{proof}
		Suppose $n$ is such that Star$(n)$ is a Player 1 win. We play the Star$(n+2)$ game with Red going first and show that Blue has a winning strategy.
		
		Consider the configurations $S_3$ and $S_4$ depicted in \Cref{fig7}. In each case, there are precisely two red edges drawn and we include the possibility that the blue edge $u_1u_2$ contained in the $K_n$ might be incident to one or both of the red edges (i.e. we might have $u_i=w_j$ for some $i$ and $j$), but we require that the two vertices $w_1$ and $w_2$ be distinct.
		
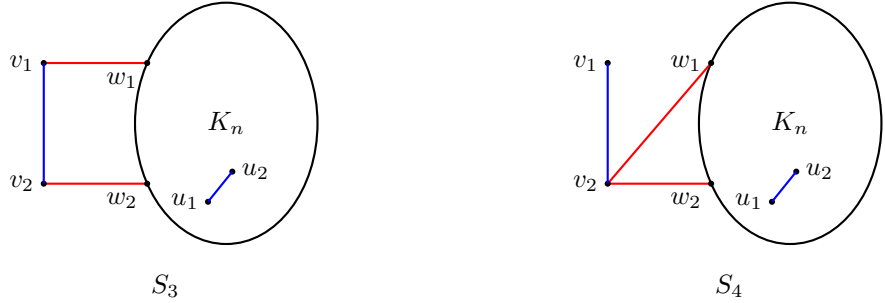
\begin{figure}[htb]
			\captionsetup{justification=centering}
			\begin{center}
				\begin{tikzpicture}[scale=0.8, cross/.style={path picture={ 
							\draw[black]
							;
					}}]
					\def\a{1.5} 
					\def\b{2} 
					\def\q{4} 
					\def\x{{\a^2/\q}} 
					\def\y{{\b*sqrt(1-(\a/\q)^2}} 
					\coordinate (O) at (-1,0); 
					\coordinate (Q) at (-\q,-1); 
					\coordinate (Q') at (-\q, 1);
					\coordinate (P) at (-\x,\y); 
					\coordinate (P') at (-\x,-\y);
					
					\draw[thick] (O) ellipse({\a} and {\b});
					\draw[red, thick] (-4, 1) -- (-2.3, 1);
					
					\draw[red, thick] (-4, -1) -- (-2.3, -1);
					\fill[] (Q) circle(0.05) node[left] {$v_2$};
					\fill[] (Q') circle(0.05) node[left] {$v_1$};
					\fill[] (-2.3, 1) circle(0.05) node[below left] {$w_1$};
					\fill[] (-2.3, -1) circle(0.05) node[below left] {$w_2$};
					\draw[blue, thick] (-4, -1) -- (-4, 1);
					\fill[] (-1.3, -1.3) circle(0.05) node[left] {$u_1$};
					\fill[] (-0.9, -0.8) circle(0.05) node[right] {$u_2$};
					\draw[blue, thick] (-1.3, -1.3) -- (-0.9, -0.8);
					
					\node at (-1, 0) {$K_n$};
					\node at (-2, -2.7) {$S_3$};
				\end{tikzpicture}
				\hspace{3cm}
				\begin{tikzpicture}[scale=0.8, cross/.style={path picture={ 
							\draw[black]
							;
					}}]
					\def\a{1.5} 
					\def\b{2} 
					\def\q{4} 
					\def\x{{\a^2/\q}} 
					\def\y{{\b*sqrt(1-(\a/\q)^2}} 
					\coordinate (O) at (-1,0); 
					\coordinate (Q) at (-\q,-1); 
					\coordinate (Q') at (-\q, 1);
					\coordinate (P) at (-\x,\y); 
					\coordinate (P') at (-\x,-\y);
					
					\draw[thick] (O) ellipse({\a} and {\b});
					\draw[red, thick] (-4, -1) -- (-2.3, 1);
					
					\draw[red, thick] (-4, -1) -- (-2.3, -1);
					\fill[] (Q) circle(0.05) node[left] {$v_2$};
					\fill[] (Q') circle(0.05) node[left] {$v_1$};
					\draw[blue, thick] (-4, -1) -- (-4, 1);
					\fill[] (-1.3, -1.3) circle(0.05) node[left] {$u_1$};
					\fill[] (-0.9, -0.8) circle(0.05) node[right] {$u_2$};
					\draw[blue, thick] (-1.3, -1.3) -- (-0.9, -0.8);
					\fill[] (-2.3, 1) circle(0.05) node[left] {$w_1$};
					\fill[] (-2.3, -1) circle(0.05) node[below left] {$w_2$};
					
					\node at (-1, 0) {$K_n$};
					\node at (-2, -2.7) {$S_4$};
				\end{tikzpicture}
				\caption{The positions $S_3$ and $S_4$, respectively.}\label{fig7}
				
			\end{center}
		\end{figure}
		\begin{claim*}
			If Red and Blue alternate claiming one edge starting from either $S_3$ or $S_4$ with Red going first, then Blue has a strategy to end the game with a star of size at least as large as Red's largest star.
		\end{claim*}
		\begin{subproof}
			Note that the restriction of both $S_3$ and $S_4$ to the $K_n$ subgraph $K_{n+2}\setminus\{ v_1, v_2\}$ pictured is precisely an instance of the Star$(n)$ game in which Blue has made the first move, and no other move has been played yet. Blue now plays as follows.
			
			Whenever Red makes a move inside $K_n$, Blue responds according to the winning strategy of Player 1 for the Star$(n)$ game. If Red has just claimed the last free edge of $K_n$, Blue skips his turn. For all other moves, Blue plays according to an explicit pairing strategy for all of the unclaimed edges incident to one of $v_1$ and $v_2$. More specifically, we pair up these unclaimed edges as $(v_1x, v_2x)$ for each $x\in K_n\setminus \{w_1, w_2\}$ together with either the pair $(v_1w_2, v_2w_1)$ (if we started from $S_3$) or $(v_1w_1, v_1w_2)$ (if we started from $S_4$). 
			Blue then plays so that whenever Red claims one element of any of these pairs, Blue responds by claiming the other element from the same pair.
			
			Let $G_R,G_B$ be the final graphs claimed at the end of the game and $G_R',G_B'$ the subgraphs on $K_{n+2}\setminus\{v_1,v_2\}$.  It is not difficult to check that regardless of whether we started from $S_3$ or $S_4$, at the end of the game we have $\deg_{G_R}(v_1)=\deg_{G_B}(v_2)$ (and therefore also $\deg_{G_B}(v_1)=\deg_{G_R}(v_2)$), so neither $v_1$ nor $v_2$ can be the centre of a winning star for Red (as Player 1). Now, the crucial observation is that at the end of the game, Blue attains a Player 1 winning position inside the $K_n$ sub-board, and moreover, every vertex $x\in K_n$ except for exactly one of the vertices $w_1$ or $w_2$ sees precisely one red edge and one blue edge among the edges $v_1x$ and $v_2x$. Write $y\in \{w_1,w_2\}$ for the vertex in $K_n$ which does not satisfy this property, and let $z\in K_n$ be such that it maximizes $\deg_{G_B'}(z)$.
			
			\textit{Case 1:} $y\neq z$. In this case we have by the definition of $y$ that $z$ has exactly one blue neighbor in $\{v_1,v_2\}$, and hence
			\[\deg_{G_B}(z)=\deg_{G'_B}(z)+1> \max_{x\in K_n\setminus\{y\}}\deg_{G_R'}(x)+1=\max_{x\in K_n\setminus\{y\}}\deg_{G_R}(x)\] and also \[\deg_{G_B}(z)=\deg_{G'_B}(z)+1\geq \deg_{G'_R}(y)+2=\deg_{G_R}(y)\]
			so we conclude that $\deg_{G_B}(z)\geq \max_{x\in K_n} \deg_{G_R}(x)$. Combining this with the above observation that $\deg_{G_R}(v_1)=\deg_{G_B}(v_2)$, we see that the largest blue-degree in the game is at least as large as the largest red-degree in this case.
			
			\textit{Case 2:} $y=z$. In this case,
			\[\deg_{G_B}(z)=\deg_{G_B'}(z)\geq  \max_{x\in K_n\setminus\{z\}}\deg_{G_R'}(x)+1=\deg_{G_R}(x).\]
			It remains then to verify $\deg_{G_B}(z)\ge \deg_{G_R}(z)$, which we do by further breaking our argument into subcases based on the parity of $n$.
			
			If $n=2k$, then exactly as we argued in Case 1 of \Cref{star+1}, we find that we must have $\deg_{G'_B}(z)\geq k+1$ since otherwise all vertices of $K_n$ would have blue degree $k$ and red degree $k-1$, contradicting the fact that Red has claimed half the edges of the $K_n$. This implies $\deg_{G'_R}(z)\le k-2$, and hence in particular \[\deg_{G_B}(z)=\deg_{G_B'}(z)\geq \deg_{G_R'}(z)+2=\deg_{G_R}(z).\]
			
			If instead $n=2k+1$, then here we must have $\deg_{G_B'}(z)\geq k+1$ since $\deg_{G_B'}(z)>\deg_{G_R'}(z)$, so $\deg_{G_R'}(z)\le k-1$ and hence
			\[\deg_{G_B}(z)=\deg_{G_B'}(z)\geq \deg_{G_R'}(z)+2=\deg_{G_R}(z).\]
			
			We conclude that regardless of the parity of $n$ we always have that $\deg_{G_B}(z)\geq \max_{x\in K_n} \deg_{G_R}(x)$. Again combining this with the fact that $\deg_{G_R}(v_1)=\deg_{G_B}(v_2)$, we conclude that the largest blue-degree in the game that has just been played is at least as large as the largest red-degree.
		\end{subproof}
		
		Using the above claim, it remains to show that Blue can always arrive in one of $S_3$ or $S_4$ while beginning the Star$(n+2)$ game as Player 2.  To do this, Blue will play his first edge incident to the first edge played by Red, and then after Red's second move we will end up in one of the 5 (unlabelled) diagrams depicted in \Cref{fig:5States}.  In \Cref{fig8} below we show how one can label each of these 5 cases (possibly after introducing an isolated vertex) so as to be contained in a subgraph of either $S_3$ or $S_4$.
		
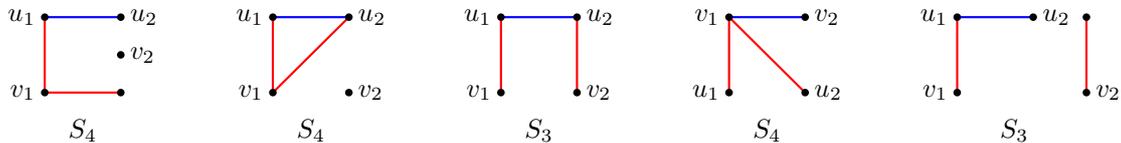
\begin{figure}[htb]
			\captionsetup{justification=centering}
			\begin{center}
				\begin{tikzpicture}
					\draw[red, thick] (1, 0) -- (0, 0) -- (0, 1);
					\draw[blue, thick] (0, 1) -- (1, 1);
					\fill[] (0, 1) circle(0.05) node[left] {$u_1$};
					\fill[] (1, 1) circle(0.05) node[right] {$u_2$};
					\fill[] (0, 0) circle(0.05) node[left] {$v_1$};
					\fill[] (1, 0.5) circle(0.05) node[right] {$v_2$};
					\fill[] (1, 0) circle(0.05);
					\node at (0.5,-0.5) {$S_4$};
					
					\draw[red, thick] (3, 1) -- (3, 0) -- (4, 1);
					\draw[blue, thick] (4, 1) -- (3, 1);
					\fill[] (3, 1) circle(0.05) node[left] {$u_1$};
					\fill[] (4, 1) circle(0.05) node[right] {$u_2$};
					\fill[] (3, 0) circle(0.05) node[left] {$v_1$};
					\fill[] (4, 0) circle(0.05) node[right] {$v_2$};
					\node at (3.5,-0.5) {$S_4$};
					
					\draw[red,thick] (6, 1)--(6, 0);
					\draw[red, thick] (7, 1)--(7, 0);
					\draw[blue,thick] (6,1)--(7,1);
					\fill[] (6, 1) circle(0.05) node[left] {$u_1$};
					\fill[] (7, 1) circle(0.05) node[right] {$u_2$};
					\fill[] (6, 0) circle(0.05) node[left] {$v_1$};
					\fill[] (7, 0) circle(0.05) node[right] {$v_2$};
					\node at (6.5,-0.5) {$S_3$};
					
					\draw[red, thick] (10, 0) -- (9, 1);
					\draw[red,thick] (9, 0) -- (9, 1);
					\draw[blue, thick] (10, 1)-- (9, 1);
					\fill[] (9, 1) circle(0.05) node[left] {$v_1$};
					\fill[] (10, 1) circle(0.05) node[right] {$v_2$};
					\fill[] (9, 0) circle(0.05) node[left] {$u_1$};
					\fill[] (10, 0) circle(0.05) node[right] {$u_2$};
					\node at (9.5,-0.5) {$S_4$};
					
					\draw[red, thick] (13.7, 0) -- (13.7, 1);
					\draw[red, thick] (12,1) -- (12,0);
					\draw[blue, thick] (12, 1) -- (13, 1);
					\fill[] (12, 1) circle(0.05) node[left] {$u_1$};
					\fill[] (13, 1) circle(0.05) node[right] {$u_2$};
					\fill[] (12, 0) circle(0.05) node[left] {$v_1$};
					\fill[] (13.7, 0) circle(0.05) node[right] {$v_2$};
					\fill[] (13.7, 1) circle(0.05) node[right] {};
					\node at (12.75,-0.5) {$S_3$};
				\end{tikzpicture}
				\caption{The five possible game states depicted in \Cref{fig:5States} together with an appropriate choice for the vertices $u_1,u_2,v_1,v_2$.}\label{fig8}
			\end{center}
		\end{figure}
		
		With this, Blue now plays his second edge to be whichever one of $v_1v_2$ or $u_1u_2$ he has not already been played. The game is then precisely in state $S_3$ or $S_4$ with Red to move, so the result follows by the above claim.
	\end{proof}
	
	We again close this section by formally putting these results together to prove our main result for this section.
	\begin{proof}[Proof of \Cref{star}]
		By the propositions above, for each integer $n\ge 3$, at least two of the numbers in $\{n,n+1,n+2\}$ are such that $\Star(n)$ is a Player 2 win.  It follows then that the set of winning values for Player 2 has asymptotic density at least $2/3$.
	\end{proof}
	
	\section{The biased clique-building game}
	
	In this section we prove \Cref{bias}, which we recall says that Player 2 wins the Clique$_{(p,q)}(n)$ game for all $q,n$ sufficiently large in terms of $p$.  Roughly speaking, we will prove this by having Player 2 dedicate $s$ of their $q$ edges each round towards blocking Player 1 from making a large clique, while simultaneously using their remaining $t:=q-s$ edges to build as large of a clique as possible for themselves\footnote{Our original proof had Player 2 dedicate all $q$ of their edges towards making Player 1's largest clique as small as possible.  This new approach, essentially due to an anonymous referee, significantly improves upon our previous bounds by dedicating more edges towards building Player 2's own clique.}.

	For this proof approach to work, we need to argue both that Player 2 can stop Player 1 from making a large clique, and that they can build a large clique themselves.  Each of these parts relies on existing results in the literature on Maker-Breaker games.  The first is a biased form of the Erd\H{o}s-Selfridge theorem due to Beck, which we have specialized below to the particular case of building cliques of a given size.
	\begin{theorem}[Beck {\cite{Beck82}}]\label{erdosselfridge}
		Suppose two players take turns claiming previously unclaimed edges of a $K_n$, with Player 1 claiming $p$ edges and Player 2 claiming $s$ edges every turn. If $k$ is such that
		\beq{beckcond} \binom{n}{k} (s+1)^{-\binom{k}{2}/p}<\frac{1}{s+1},\enq
		then Player 2 can play in a way which guarantees that at the end of the game, Player 1 has not claimed all of the edges of a clique on $k$ vertices.
	\end{theorem}
	
	We also need the following result of Gebauer~\cite{gebauer}, whose notation we have altered to better match our forthcoming strategy stealing argument.
	\begin{theorem}[\cite{gebauer}]\label{gebauer}
		Suppose two players take turns claiming previously unclaimed edges of a $K_{n}$, with Player 1 claiming $p$ edges and Player 2 claiming $t$ edges every turn.  Then Player 2 can play in a way which guarantees that at the end of the game, Player 2 has claimed all of the edges of a clique on
		\[\left(\frac{t}{\log_2(p+1)}-o(1)\right) \log_2 (n-2p)\]
		vertices, where here the $o(1)$ term tends towards 0 as $n$ tends towards infinity.
	\end{theorem}
	We note that the statement of this result differs slightly from that of \cite{gebauer}, in that the original refers to Player 1 claiming a large clique and the $\log_2(n-2p)$ term is replaced by $\log_2(n)$.  The present statement follows from the original simply by having Player 2 ignore the at most $2p$ vertices that Player 1 touches in their first move and then having Player 2 follow the first player's strategy from \cite{gebauer} to construct this large clique.

	Combining these two results and our outlined strategy above gives the following.
	\begin{proposition}\label{prop:bias}
		If $p,s,t\ge 1$ are integers satisfying
		\[t\log_2(s+1)>2p\log_2(p+1),\]
		then there exists an integer $n_0$ such that Clique$_{(p,s+t)}(n)$ is a Player 2 win for all $n\ge n_0$.
	\end{proposition}
	\begin{proof}
		As usual we let Red and Blue denote the two players of $\Clique_{(p,s+t)}(n)$ with Red going first.  For each of his turns, Blue will use $s$ of his $s+t$ edges to play according to the strategy of Player 2 given in \Cref{erdosselfridge} with 
		\[k:=1+\lceil 2p\log_{s+1}(n)\rceil.\]
		Note that this is a valid choice of $k$ provided $n$ is large enough to ensure $k\ge s+2$, since in this case 
		\[{n\choose k}\cdot (s+1)^{-{k\choose 2}/p}\le k^{-1} n^{k}\cdot (s+1)^{-{k\choose 2}/p}\le \frac{1}{s+2} (s+1)^{k \log_{s+1}(n)-{k\choose 2}/p}\le  \frac{1}{q+2},\]
		with this last step using $\log_{s+1}(n)\le \frac{1}{2p}(k-1)$ by definition of $k$.  In total, this ensures that regardless of Red's strategy, she will end the game without claiming all of the edges of a clique on $k$ vertices.  
		
		It remains to argue that Blue can use their remaining $t$ edges, whose use we have not yet specified, to construct a clique strictly larger than this.  For this, he uses his remaining $t$ edges each round to follow the strategy of Player 2 given in \Cref{gebauer}. In doing so, Blue is guaranteed to end the game having claimed all of the edges of a clique with $\left(\frac{t}{\log_2(p+1)}-o(1)\right) \log_2(n-2p)$ vertices.  By our hypothesis on $p,s,t$, we have for $n$ sufficiently large that
		\[\left(\frac{t}{\log_2(p+1)}-o(1)\right) \log_2(n-2p)\ge \frac{2p}{\log_2(s+1)}\log_{2}(n)+2\ge k.\]
		
		That is, Blue can guarantee that their largest clique is strictly larger than Red's largest clique for $n$ sufficiently large, proving the result.
	\end{proof}
	It remains now to optimize the choices of $s,t$ in \Cref{prop:bias}.  For example, a simple choice of parameters gives the following.
	
	\begin{corollary}\label{cor:bias}
		If $p,q\ge 1$ are integers satisfying $q\ge 3p+1$, then there exists an integer $n_0$ such that Clique$_{(p,q)}(n)$ is a Player 2 win for all $n\ge n_0$.
	\end{corollary}
	\begin{proof}
		Observe that the integers $s=p$ and $t=2p+1$ satisfy $t\log_2(s+1)>2p\log_2(p+1)$.  Therefore, \Cref{prop:bias} implies the result when $q=3p+1$, and the result holds for larger $q$ by monotonicity.
	\end{proof}
	
	A slightly more careful optimization allows us to prove the result for $q=(2+o(1))p$, which is the optimal range one can prove things for using \Cref{prop:bias} alone. 
	
	\begin{proof}[Proof of \Cref{bias}]
		For notational convenience, let \[\ep:=\frac{4\log_2\log_2(p+1)}{\log_2(p+1)}.\]  Recall that \Cref{bias} claims that for all $p\ge 1$ and integers
		\[q\ge (2+3\ep)p+2,\]
		the game Clique$_{(p,q)}(n)$ is a Player 2 win for all $n$ sufficiently large.  If $\ep> 1$ (which does hold for some small values of $p$), then this result immediately follows from \Cref{cor:bias}, so we may assume from now on that $p$ is such that $\ep\le 2$. Similarly if $p=1$ then $\ep=0$ and the result again follows from \Cref{cor:bias}, so we may assume $p\ge 2$.   With this these two assumptions in mind, we aim to show the following.
		\begin{claim*}
			The real numbers $s:=\ep(p+1)-1$ and $t:=(2+\ep)p$ satisfy $s,t>0$ and \[t\log_2(s+1)>2p\log_2(p+1).\]
		\end{claim*}
		\begin{proof}
			We have $t>0$ for all values of $p\ge 1$, and it is straightforward to check that $s>0$ holds due to our assumption that $p\ge 2$, so it remains only to prove the last inequality.  Note that
			
			\[t\log_2(s+1)=(2+\ep)p (\log_2(p+1)+\log_2(\ep))=2p\log_2(p+1)\cdot \left(1+\frac{\ep}{2}\right)\left(1+\frac{\log_2(\ep)}{\log_2(p+1)}\right),\]
			so it remains to show that the second term in this product is strictly greater than 1, i.e.\ that 
			\[1+\frac{\ep}{2}+\frac{\log_2(\ep)}{\log_2(p+1)}+\frac{\ep}{2}\cdot \frac{\log_2(\ep)}{\log_2(p+1)}>1.\]
			Using our assumption that $\ep\le 2$ on the last term in this sum yields that it suffices to show
			\[\frac{\ep}{2}> -\frac{2 \log_2(\ep)}{\log_2(p+1)},\]
			or equivalently that
			\[\log_2(p+1)>4 \ep^{-1}\cdot  \log_2(\ep^{-1})= \frac{\log_2(p+1)}{\log_2\log_2(p+1)}\cdot \left(\log_2\log_2(p+1)-\log_2(4\log_2\log_2(p+1))\right),\]
			and this does indeed hold since $\log_2(4\log_2\log_2(p+1))>0$ for $p\ge 2$.
		\end{proof}
		With the above claim in mind, the integers $\lceil s\rceil$ and $\lceil t\rceil$ are both positive and satisfy the conditions of \Cref{prop:bias}.  As such, we conclude that Player 2 wins Clique$_{(p,q)}(n)$ for $n$ large and $q$ at least
		\[\lceil s\rceil+\lceil t\rceil\le s+t+2=(2+2\ep)p+\ep+1\le (2+3\ep)p+2,\]
		proving the result.
	\end{proof}
	
	\section{Concluding remarks}
	In this paper, we showed that in the games $\Clique(n)$ and Star$(n)$ introduced by Erd\H{o}s, Player 2 wins for most possible values of $n$. In order to show that Player 2 wins for all $n$ sufficiently large, it is natural to try and prove this for a more general set of games.  There are many such games one could consider; we mention one natural case here. 
	
	\begin{definition*}
		Given a graph $H$, we define the two-player game Clique$(H)$ as follows. Two players take turns claiming previously unclaimed edges of $H$ until all edges have been claimed. They each claim precisely one edge on each of their respective turns, with Player 1 going first. Let $G_1$ and $G_2$ be the two graphs claimed by Player 1 and Player 2, respectively, at the end of the game (so that $E(G_2)=E(H)\setminus E(G_1)$). Player 1 wins the game if and only if they have claimed a strictly larger clique than Player 2, i.e. if $\omega(G_1)>\omega(G_2)$.
	\end{definition*}
	Note that in this definition, $\Clique(K_n)=\Clique(n)$.  For this more general version of the game, a natural followup question to the conjecture of Erd\H{o}s in \Cref{intro} would be: is $\Clique(H)$ a Player 2 win for all $H\ne K_2$? A little bit of thought yields that this is not the case.
	\begin{proposition}\label{prop:genHost}
		If $H_t$ is the graph which consists of $t\ge 2$ copies of $K_4$ sharing a common edge $uv$, then Player 1 wins the Clique$(H_t)$ game.
	\end{proposition}
	\begin{proof}[Sketch of Proof]
		Let $C_1,\ldots,C_t$ denote the copies of $K_4$ in $H_t$.  In this game, Player 1 starts by claiming the edge $uv$, after which Player 1 allows Player 2 to claim one edge from each $C_i\ne C_1$, and then Player 1 claims the edge in $C_1$ opposite $uv$.  We then claim that Player 1 can play such that (1) whenever Player 2 claims an edge in $C_i$, Player 1 also claims an edge in $C_i$, (2) Player 2 never manages to claim a triangle, and (3) Player 1 claims a triangle in $C_1$.  And indeed, a little bit of case analysis shows that it is possible to play while ensuring all of these conditions hold, giving the result.
	\end{proof}
	It is possible to extend this class of examples, and in particular it is possible to find $d$-regular graphs $H$ for any $d\ge 5$ such that $\Clique(H)$ is a Player 1 win\footnote{It suffices to show that there exist $d$-regular graphs $H$ such that (1) $H_t\subseteq H$ for some $t\ge 2$ and (2) every edge of $H-H_t$ is not contained in a triangle of $H$ (since in this case Player 1 can simply give every edge of $H-H_t$ to Player 2 at the start and play their winning strategy in $H_t$ to create a triangle).  Such an $H$ can be constructed by taking $H_t$ together with a union of copies of $K_{d-1,d-1}$ such that every part in some $K_{d-1,d-1}$ is adjacent to a vertex in $H_t$ in an appropriate way.}.  Still, it is natural to ask if there exist a ``richer'' set of counterexamples.
	\begin{question}
		Do there exist graphs $H$ with arbitrarily large clique number such that $\Clique(H)$ is a Player 1 win?  Do there exist infinitely many edge-transitive graphs such that $\Clique(H)$ is a Player 1 win?
	\end{question}
	
	Of course, one can also introduce the games Star$(H)$ and $\Clique_{(p,q)}(H)$ defined analogously as above for $\Clique(H)$.  Again one can easily construct infinitely many examples where Player 1 wins these games (e.g. $H=K_{1,2t-1}$ for Star$(H)$ and any triangle-free graph for $\Clique_{(p,q)}(H)$ with $p<q$), and again it is natural to ask if there exists a ``richer'' set of counterexamples.  In particular, the following would be of interest.
	\begin{question}
		Do there exist infinitely many regular graphs $H$ such that $\Star(H)$ is a Player 1 win?
	\end{question}
	
	For the biased clique-building game $\Clique_{(p,q)}(n)$, we proved that Player 2 wins for all $q\ge (2+o(1))p$ which, by the same token, implies that Player 1 wins for all $p\ge (2+o(1))q$.  We are inclined to believe that the following should hold.
	
	\begin{conjecture}\label{conj1}
		If $p>q$ then Player 1 wins $\Clique_{(p,q)}(n)$ for all $n$ sufficiently large.  If $q \ge p$ then Player 2 wins $\Clique_{(p,q)}(n)$ for all $n$ sufficiently large.
	\end{conjecture}
	That is, if a player is allowed to make more moves than their opponent during the game, then they can construct a strictly larger clique than their opponent.

	Finally, we note that just as we defined the biased clique-building game $\Clique_{(p,q)}(n)$ as a generalization of $\Clique(n)$, one can also define an analogous biased star-building game $\Star_{(p,q)}(n)$ in a similar way.  Here the most natural question to ask is whether $\Star_{(1,q)}(n)$ is a Player 2 win for all $q\ge 2$ and $n$ sufficiently large, and indeed this is easy to prove: if Player 1 claims an edge $uv$, then Player 2 can simply use their next turn to claim edges incident to $u$ and $v$.  The more general games $\Star_{(p,q)}(n)$ would also be of interest to study further.
	
	\section*{Acknowledgements}
	We would like to thank Thomas Bloom for bringing this problem to our attention and providing the original reference.  We would also like to thank the various referees who greatly improved the presentation of this article, and in particular for their comments leading to a substantial improvement to the bounds of \Cref{bias}.
	
	\bibliographystyle{plain}
	\bibliography{bibliography}

\begin{thebibliography}{1}

\bibitem{Beck82}
J.~Beck.
\newblock Remarks on positional games. {I}.
\newblock {\em Acta Mathematica Academiae Scientiarum Hungaricae},
  40(1):65--71, 1982.

\bibitem{beck}
J.~Beck.
\newblock {\em Combinatorial {G}ames: {T}ic-{T}ac-{T}oe {T}heory}.
\newblock Encyclopedia of Mathematics and Its Applications 114. Cambridge
  University Press, 2008.

\bibitem{gebauer}
Heidi Gebauer.
\newblock On the clique-game.
\newblock {\em European Journal of Combinatorics}, 33(1):8--19, 2012.

\bibitem{guy83}
R.~K. Guy.
\newblock A {M}iscellany of {E}rd{\H{o}}s {P}roblems.
\newblock {\em The American Mathematical Monthly}, 90(2):118--120, 1983.

\bibitem{HKSSz}
D.~Hefetz, M.~Krivelevich, M.~Stojakovi\'{c}, and T.~Szab\'{o}.
\newblock {\em Positional {G}ames}.
\newblock Oberwolfach Seminars. Birkh\"{a}user Basel, 2014.

\bibitem{Zermelo}
E.~Zermelo.
\newblock {\"U}ber eine {Anwendung} der {Mengenlehre} auf die {Theorie} des
  {Schachspiels} (in {G}erman).
\newblock In E.~W. Hobson and A.~E.~H. Love, editors, {\em Proceedings of the
  Fifth International Congress of Mathematicians (Cambridge 1912)}, volume~2,
  pages 501--504, Cambridge, 1913. Cambridge University Press.

\end{thebibliography}

\end{document}